\pdfoutput=1 
\documentclass[12pt]{amsart}

\usepackage{enumerate}

\usepackage[active]{srcltx}
\usepackage{nicefrac}

\usepackage{pgf}
\usepackage{tikz}
\usetikzlibrary{arrows,automata}
\usepackage{graphicx}
\usetikzlibrary{shapes,patterns,calc,snakes}

\usepackage{amsthm,amssymb,setspace}
\usepackage[pagebackref,colorlinks,linkcolor=red,citecolor=blue,urlcolor=blue,hypertexnames=true]{hyperref}
\usepackage{amsrefs}
\usepackage[matrix, arrow]{xy}

\renewcommand{\deg}{{\rm deg}}

\newcommand{\A}{\mathcal A}

\newcommand{\N}{\mathbb N}
\newcommand{\R}{\mathbb R}

\newcommand{\B}{\mathcal B}

\theoremstyle{plain}
\newtheorem*{theorem*}{Theorem}
\newtheorem{theorem}{Theorem}[section]

\newtheorem{lemma}[theorem]{Lemma}

\theoremstyle{definition}
\newtheorem*{definition*}{Definition}
\newtheorem{definition}[theorem]{Definition}

\theoremstyle{remark}

\newtheorem{remark}[theorem]{Remark}
\newtheorem{example}[theorem]{Example}

\begin{document}

\onehalfspace

\title{Positivstellens\"atze for Quantum Multigraphs}

\author{Tim Netzer}
\address{Tim Netzer, Universit\"at Leipzig, Germany}
\email{netzer@math.uni-leipzig.de}

\author{Andreas Thom}
\address{Andreas Thom, Universit\"at Leipzig, Germany}
\email{thom@math.uni-leipzig.de}


\begin{abstract}Studying inequalities between subgraph- or homomorphism-densities is an important topic in graph theory. Sums of squares techniques have proven useful in dealing with such questions. Using an approach from real algebraic geometry, we strengthen a Positivstellensatz for simple quantum graphs by Lov\'asz and Szegedy, and we prove several new Positivstellens\"atze for nonnegativity of quantum  multigraphs. We provide new examples and counterexamples.
\end{abstract}

\maketitle

\tableofcontents

\section{Introduction} Let $F,G$ be finite undirected graphs without multiple edges or loops (all graphs in the first part of this paper are of this type). A homomorphism is a mapping $\varphi\colon V_F\rightarrow V_G$ defined on vertices, which preserves the adjacency relation, i.e. whenever $ij\in E_F$ is an edge in $F$, then $\varphi(i)\varphi(j)\in E_G$ is an edge in $G$. The {\it homomorphism density} $t(F,G)$ of $F$ in $G$ is the probability that a randomly chosen map $\varphi\colon V_F\rightarrow V_G$ is a homomorphism. So if ${\rm hom}(F,G)$ denotes the number of homomorphisms, then $$t(F,G)=\frac{{\rm hom}(F,G)}{\vert V_G\vert ^{\vert V_F\vert}}.$$ The {\it subgraph density} $t_{\rm inj}(F,G)$ is closely related; it is the probability that a random {\it injective} map is a homomorphism, i.e. $$t_{\rm inj}(F,G)= \frac{{\rm inj}(F,G)}{\vert V_G\vert\cdot  (\vert V_G\vert -1)\cdots (\vert V_G\vert -\vert V_F\vert +1)},$$ where ${\rm inj}(F,G)$ is the number of injective homomorphisms. With $F$ fixed and the number of vertices of $G$ growing, $t(F,G)$ and $t_{\rm inj}(F,G)$ coincide asymptotically, as for example shown in \cite{lose1}.  Since these densities are often studied in the context of very large graphs $G$, information about any of the two densities also contains some information about the other. We will mostely be concerned with the homomorphism density $t(\cdot,\cdot)$ in this paper. 

One is interested in the possible values that can occur as homomorphism densities, and the relations between them. In other words, given graphs $F_1,\ldots,F_n$, one wants to understand the set  $$\left\{ (t(F_1,G),\ldots,t(F_n,G))\mid G \mbox{ graph }\right\}\subseteq \R^n$$ (see \cite{bochloso} Section 7.3 for a nice picture in the case $n=2, F_1=K_2, F_2=K_3$).
A way of doing this is looking at {\it polynomial inequalities} between homomorphism densities. Given a polynomial $p\in \R[x_1,\ldots,x_n]$, one is interested in the question whether $$p(t(F_1,G),\ldots,t(F_n,G))\geq 0$$ holds {\it for all graphs} $G$, i.e. whether $p$ is nonnegative on the above set.  Note that the homomorphism density is multiplicative in the first component, meaning that $$t(F_1\sqcup F_2,G)=t(F_1,G)\cdot t(F_2,G),$$ where $F_1\sqcup F_2$ denotes the disjoint union of graphs $F_1$ and $F_2$. 
So (after changing the $F_i$)  we can restrict to {\it linear} inequalities: given graphs $F_1,\ldots,F_n$ and $c_1,\ldots,c_n\in\R$, does  $$\sum_{i=1}^n c_i\cdot t(F_i,G)\geq 0$$ hold for all $ G?$  
\begin{definition}\label{defqu}
(1) A {\it quantum graph} is a formal linear combination of graphs, with real coefficients: $a=\sum_{i=1}^n c_iF_i.$

(2) A quantum graph $a=\sum_i c_iF_i$ is called {\it nonnegative} if $t(a,G):=\sum_i c_i\cdot  t(F_i,G)\geq 0$ holds for all graphs $G$. \end{definition}

\begin{example}\label{quex}
(1)  The following quantum graph is nonnegative:  

\medskip
 \begin{center}
\begin{tikzpicture}[-,>=stealth',shorten >=1pt,auto,node distance=2.8cm,
                    semithick]
  \tikzstyle{every state}=[fill,draw=none,text=white,scale=0.2]

  \node[state] (A)                    {};
  \node[state]         (B) [below left of=A] {};
  \node[state] (C) [below right of =A]{};

  \path (A) edge           node {} (B)
         edge   []        node {} (C)
               ;

  \node[state]         (D) at (2,0) {};
   \node[state]         (E) [below  of=D] {};
\node[state] (F) [right of = D]{};
\node[state] (G) [below of = F]{};
\path
		(D) edge   []        node {} (E)
  		 (F) edge node {} (G);             
              
    \put (30,-10){$-$}

\put (-50,-8){$a=$}
               \end{tikzpicture}
\end{center}

\noindent
This is shown in \cite{lo}, using an easy sums of squares approach; see Example \ref{ex1} below for more details.

\noindent
(2) The following quantum graph is also nonnegative; we will prove this in Example \ref{rob} below:
$$\begin{tikzpicture}[-,>=stealth',shorten >=1pt,auto,node distance=2.8cm,
                    semithick]
  \tikzstyle{every state}=[fill,draw=none,text=white,scale=0.2]

  \node[state] (A)                    {};
  \node[state]         (B) [below left of=A] {};
  \node[state] (C) [below right of =A]{};

  \path (A) edge           node {} (B)
         edge   []        node {} (C)
         (B) edge (C)
               ;

   \node[state] (D)        at (2.3,0)            {};
  \node[state]         (E) [below left of=D] {};
  \node[state] (F) [below right of =D]{};

  \path (D) edge           node {} (E)
         edge   []        node {} (F)
               ;

     \put (30,-10){$-2$}

\put (-50,-8){$b=$}

\node[state] (G) at (4,0){};
\node[state] (H) [below of =G] {};
\path (G) edge (H);
\put (95,-10){$+$}
               \end{tikzpicture}$$

\noindent
(3) The computation $c:=b+2a$ results in the following quantum graph, whose nonnegativity is known as Goodman's Theorem:

$$\begin{tikzpicture}[-,>=stealth',shorten >=1pt,auto,node distance=2.8cm,
                    semithick]
  \tikzstyle{every state}=[fill,draw=none,text=white,scale=0.2]

  \node[state] (A)                    {};
  \node[state]         (B) [below left of=A] {};
  \node[state] (C) [below right of =A]{};

  \path (A) edge           node {} (B)
         edge   []        node {} (C)
         (B) edge (C)
               ;

  \node[state]         (D) at (2,0) {};
   \node[state]         (E) [below  of=D] {};
\node[state] (F) [right of = D]{};
\node[state] (G) [below of = F]{};
\path
		(D) edge   []        node {} (E)
  		 (F) edge node {} (G);             
              
    \put (30,-10){$-2$}

\put (-50,-8){$c=$}

\node[state] (G) at (4,0){};
\node[state] (H) [below of =G] {};
\path (G) edge (H);
\put (95,-10){$+$}
               \end{tikzpicture}$$
\noindent  
 This is precisely the statement that the polynomial $y-2x^2+x$ is nonnegative on the set $$\{(t(K_2,G),t(K_3,G))\mid G \mbox{ graph}\}\subseteq \R^2.$$
  \end{example}

Nonnegativity of quantum graphs is examined in numerous recent papers. It is  in general an undecidable problem \cite{hano}, but sums of squares techniques have proven useful in attacking it \cite{lose2}. An extensive account of this topic (any many related others) can be found in the very nice book \cite{lo}.   

Our contribution is the following. By putting the existing sums of squares techniques into a bit more conceptual setting of real algebraic geometry, we  simplify and slightly strengthen the Positivstellensatz from \cite{lose2}. This is done in Section \ref{simple}. Our main results are Theorem \ref{main}, Theorem \ref{weakpos} and Theorem \ref{three} in Section \ref{multisec}, all Positivstellens\"atze for quantum {\it multigraphs}. We obtain new examples, using results from real algebraic geometry.

\section{Simple graphs}\label{simple}

In this section, every graph is finite, undirected and without multiple edges or loops. We start by explaining the setup of graph algebras and graph parameters. Let us emphasize that hardly any of the results in this section is new; the concepts have been introduced and used by several authors before (see for example \cite{frlosc,hano,lose1,lose1.5,lose2}  and also  \cite{lo} for a thorough overview). Our approach will however simplify some of the proofs, and will most notably allow us to extend the results to the multigraph setup in the next section.

 A {\it k-labeled} graph is a graph where $k$ different vertices are labeled from $1$ to $k$ (a $0$-labeled graph is an unlabeled graph). Let $\mathcal G_k$ denote the set of isomorphism classes of $k$-labeled graphs, where isomorphisms are supposed to respect the labeling. If $F,G$ are $k$-labeled graphs, then the product $$F*_k G$$ is defined as first taking the disjoint union of $F$ and $G$, then identifying vertices with the same label, and finally reducing possible edge  multiplicities to one. So for $0$-labeled graphs it is just the disjoint union. This multiplication turns $\mathcal G_k$ into an abelian monoid, having the graph $E_k$ with vertices $1,\ldots,k$ and no edges as its identity element.

The {\it $k$-th graph algebra} $\mathcal A_k$ is the monoid algebra of $\mathcal G_k$ over $\R$, i.e. it has, as a vector space, the elements of $\mathcal G_k$ as a basis: $$\A_k=\left\{ \sum_{G\in\mathcal G_k} \alpha_G \cdot G \mid  \alpha_G\in\R, \mbox{ almost all } \alpha_G=0\right\}.$$ The multiplication of $\mathcal G_k$ extends by distributivity, making $\mathcal A_k$ a commutative algebra. Note that elements of $\A_0$ are precisely quantum graphs as in Definition \ref{defqu}. 

We can equip $\mathcal A_k$ with a grading, by defining $$\deg(G):= \vert V_G\vert -k$$  (i.e. counting the unlabeled vertices) for $G\in \mathcal G_k$ and setting $$\A_k^d:=\left\{ \sum_{\deg(G)=d} \alpha_G \cdot G\right\}. $$ We obtain $$\A_k=\bigoplus_{d\geq 0} \A_k^d$$ and the multiplication is compatible with this direct-sum-decomposition, in the usual way. We will often work with the degree zero part $\A_k^0$ only. It is a finite dimensional and real reduced algebra (i.e. $0$ is a sum of squares only in the trivial way), in fact the quotient of the polynomial algebra $\R[z_{ij}\mid 1\leq i<j\leq k]$ by the ideal generated by $z_{ij}^2-z_{ij}$. Here we identify a monomial $$z^e=z_{12}^{e_{12}} \cdots z_{23}^{e_{23}} \cdots $$ (where $e_{ij}\in\{0,1\}$) with the graph having an edge  between the vertices labeled $i$ and $j$ if and only if $e_{ij}=1$. The variety  corresponding to $\A_k^0$ is finite and consists only of real points: $$\mathcal V(\A_k^0)=\{0,1\}^{k \choose 2}.$$   From this it is clear that the set of sums of squares  $\Sigma^2 \A_k^0$ in $\A_k^0$ coincides with the set of elements which are nonnegative as polynomial functions on $\mathcal V(\A_k^0)$.

To a graph in $\mathcal G_k$ we can add a new isolated vertex labeled $k+1$, and obtain a graph  in $\mathcal G_{k+1}.$ This injective monoid-homomorphism $\boxplus\colon\mathcal G_k \rightarrow \mathcal G_{k+1}$ extends to an embedding of graded algebras $\boxplus\colon\mathcal A_k\rightarrow\mathcal A_{k+1}.$

A {\it graph parameter }Êis a mapping $t\colon \mathcal G_0 \rightarrow \R,$ i.e. a rule that assigns a real number to each (unlabeled) graph. By ignoring the labels one can extend $t\colon\mathcal G_k\rightarrow \R$ for all $k$, and thus obtain linear functionals $t\colon \A_k\rightarrow \R.$
\begin{definition}A graph parameter $t$ is called \begin{itemize} 
\item  {\it isolate indifferent} if the value at a graph does not change when adding an isolated vertex; equivalently, if $t$ is compatible with the mappings $\boxplus$. 
\item {\it reflection positive} if $t(a^2)\geq 0$ holds for all $a\in\A_k$ and all $k$. 
\item {\it flatly reflection positive} if $t(a^2)\geq 0$ holds for all $a\in\A_k^0$ and all $k$. 
\end{itemize}\end{definition}

\noindent
We list some important observations and results:\begin{itemize}
\item For any graph $G$, the homomorphism density $t(\cdot, G)$ defines an   isolate indifferent and reflection positive graph parameter. The first property is obvious, the second follows for example from Remark \ref{denspos} below. 
\item Every isolate indifferent and reflection positive graph parameter is a conic combination  of limits of homomorphism densities $t(\cdot,G).$  This is shown in \cite{lose2}. So nonnegativity of quantum graphs as in Definition \ref{defqu} could also be defined as nonnegativity at each isolate indifferent and reflection positive graph parameter!
\item An isolate indifferent and flatly reflection positive graph parameter is automatically reflection positive. This is also shown in \cite{lose2}. So nonnegativity of quantum graphs as in Definition \ref{defqu} could also be defined as nonnegativity at each isolate indifferent and flatly reflection positive graph parameter!
\end{itemize}

\noindent
Now there is an obvious way to prove nonnegativity of a quantum graph $a$: if it coincides with a sum of squares from some $\A_k$ (after removing the labels and possibly adding or removing isolated vertices), then $a$ is nonnegative. 

\begin{example}\label{ex1} This example is taken from \cite{lo}. The quantum graph  

\medskip
 \begin{center}
\begin{tikzpicture}[-,>=stealth',shorten >=1pt,auto,node distance=2.8cm,
                    semithick]
  \tikzstyle{every state}=[fill,draw=none,text=white,scale=0.2]

  \node[state] (A)                    {};
  \node[state]         (B) [below left of=A] {};
  \node[state] (C) [below right of =A]{};

  \path (A) edge           node {} (B)
         edge   []        node {} (C)
               ;

  \node[state]         (D) at (2,0) {};
   \node[state]         (E) [below  of=D] {};
\node[state] (F) [right of = D]{};
\node[state] (G) [below of = F]{};
\path
		(D) edge   []        node {} (E)
  		 (F) edge node {} (G);             
              
    \put (30,-10){$-$}

\put (-50,-8){$a=$}
               \end{tikzpicture}
\end{center}

\noindent
is nonnegative, since it coincides up to labels and isolated vertices with the following square in $\A_1$:

\begin{center}
\begin{tikzpicture}[-,>=stealth',shorten >=1pt,auto,node distance=2.8cm,
                    semithick]
  \tikzstyle{every state}=[fill,draw=none,text=white,scale=0.2]

  \node[state] (A)                  {};
  \node[draw,shape=circle,scale=0.2]         (B) [below  of=A] {\huge $1$};

  \path (A) edge           node {} (B)
                     ;

  \node[state]         (D) at (2,0) {};
   \node[state]         (E) [below  of=D] {};
\node[draw,shape=circle,scale=0.2] (F) [left of = E]{\Huge $1$};
\path
		(D) edge   []        node {} (E)
  		;             
              
    \put (20,-10){$-$}

\path (-0.2,-0.9) edge [bend left =30] (-0.2,0.4);
\path (2.2,-0.9) edge [bend right =30] (2.2,0.4) ;
\put (70,5){$2$};
               \end{tikzpicture} 
\end{center}
\end{example}

\noindent
The Positivstellensatz from \cite{lose2} states that any nonnegative quantum graph arises in this way, {\it up to an arbitrarily small error $\epsilon$ in the $\ell_1$-norm of coefficients}. Note that \cite{lo2} provides a Positivstellensatz without errors, using infinite sums of squares instead.
We give a new proof for the following  strong approximation result:
\begin{theorem}\label{mainsimple} A quantum graph $a$  is nonnegative if and only if for all $\epsilon >0$ there is some $k$ and a sum of squares $\sigma\in\Sigma^2 \A_k^0,$ such that $a+\epsilon $ and $\sigma$ coincide up to labels and isolated vertices. 
\end{theorem}

The proof of the theorem becomes quite easy, if we equip our graph algebras with some more structure. So first note that permutation of the labels yields an operation $S_k\curvearrowright \mathcal G_k$ of the symmetric group $S_k$ by automorphisms  on $\mathcal G_k$. 
This operation extends to an operation by graded algebra automorphisms on $\mathcal A_k$. We denote by $\mathcal{B}_k$ the set of invariant elements of this action.  $\mathcal B_k$ is a graded subalgebra of $\mathcal A_k$, and the inclusion $\mathcal B_k\subseteq \mathcal A_k$ admits a left-inverse $\mathcal B_k$-module homomorphism \begin{align*}r\colon \mathcal A_k&\rightarrow \mathcal B_k \\ a&\mapsto \frac{1}{\vert S_k\vert}\sum_{\sigma\in S_k} a^\sigma\end{align*} which respects the grading, the {\it Reynolds operator}.  As a vector space, $\mathcal B_k$ is spanned by the elements $r(G)$ with $G\in\mathcal G_k.$  $\mathcal B_k^0\subseteq \A_k^0$ is a  subalgebra, which is clearly also finite dimensional and real reduced. The variety of $\mathcal B_k^0$ consists of finitely many points which are all real, and to the inclusion $\mathcal B_k^0\subseteq\mathcal \A_k^0$ there corresponds a surjective  polynomial mapping $\mathcal V(\A_k^0)\rightarrow \mathcal V(\mathcal B_k^0)$.
We also obtain  injective linear maps $$\boxplus_r:= r\circ \boxplus \colon \mathcal B_k\rightarrow \mathcal B_{k+1}$$ making the following diagram commutative: $$\xymatrix{\cdots  \ar@{->}^{\boxplus}[r] & \A_k  \ar@{->}^{r}[d] \ar@{->}^{\boxplus}[r] & \A_{k+1}  \ar@{->}^{r}[d]  \ar@{->}^{\boxplus}[r] & \cdots\\ \cdots  \ar@{->}^{\boxplus_r}[r] & \mathcal B_k   \ar@{->}^{\boxplus_r}[r] & \mathcal B_{k+1}  \ar@{->}^{\boxplus_r}[r] & \cdots}$$ Note that the mappings $\boxplus_r$ are just linear, not multiplicative; they are however  compatible with the grading on $\mathcal B_k$, and we often consider the degree zero part of the above diagram only. We denote by $\mathcal B^0$ the direct limit of the chain $$\cdots\rightarrow \mathcal B_k^0\rightarrow\mathcal B_{k+1}^0\rightarrow\cdots$$ in the category of $\R$-vector spaces.
We next consider $$\mathcal C_k:= r(\Sigma^2 \A_k^0)=\Sigma^2\A_k^0\cap\mathcal B_k^0.$$ From the fact that the mapping $\mathcal V(\A_k^0)\rightarrow\mathcal V(\mathcal B_k^0)$ is surjective we see that $\mathcal C_k$ is the set of nonnegative functions on $\mathcal V(\mathcal B_k^0)$, and thus also coincides with $\Sigma^2\mathcal B_k^0$ (a fact which is not true for Reynolds operators of group actions in general!). Clearly, $1$ is an interior point of the convex cone $\mathcal C_k$ in $\mathcal B_k^0$, meaning that $1+\epsilon b$ belongs to $\mathcal C_k$, for each $b\in \mathcal B_k^0$ and $\epsilon>0$ small enough. We have $\boxplus_r(\mathcal C_{k})\subseteq \mathcal C_{k+1}$.  In the direct limit $\mathcal B^0$ we obtain the convex cone $\mathcal C:=\bigcup_k \mathcal C_k,$ of which $1$ is also an interior point.
Since a graph parameter $t$ ignores labels, it factors through $\B_k$ via $r$: $$\A_k\stackrel{r}{\rightarrow}\mathcal B_k\stackrel{t}{\rightarrow}\R.$$ 

\begin{lemma} \label{help}  A family of linear functionals $\varphi\colon \mathcal B_k^0\rightarrow \R$ (for all $k\geq 0$) comes from a flatly reflexion positive and isolate indifferent graph parameter if and only if it is compatible with the embeddings $\boxplus_r$ and satisfy $\varphi(\mathcal C_k)\geq 0$ for all $k$ (equivalently, if it comes from a linear functional on $\mathcal B^0$ which is nonnegative on $\mathcal C$).\end{lemma}
\begin{proof} Easy exercise. \end{proof}

\begin{proof}[Proof of Theorem \ref{mainsimple}] One direction is clear. For the other, let $a$ be a nonnegative quantum graph. Choose some $\tilde a\in\A_d^0$ which coincides with $a$ up to isolated vertices when all labels are removed. Then $b:=r(\tilde a)\in \mathcal B_d^0$ also coincides with $a$ up to isolated vertices and labels, and is thus nonnegative at each isolate indifferent, flatly reflection positive graph parameter.  In view of Lemma \ref{help}, it belongs to the double dual of $\mathcal C$ in $\mathcal B^0$, and the isolation theorem for convex sets with nonempty interior (see for example \cite{cimane}, Proposition 1.3 for this standard fact) implies $b+\epsilon \in \mathcal C$ for all $\epsilon>0$. Since $\mathcal B^0$ is the direct limit of the $\mathcal B_k^0$ and $\mathcal C=\bigcup_k \mathcal C_k$, this proves the claim.
\end{proof}

\begin{remark}
The proof even shows that if $a$ is {\it strictly} positive at each nontrivial, isolate indifferent and reflection positive graph parameter, then $a$ coincides with a sum of squares from some $\A_k^0$ {\it without error} (see again \cite{cimane}, Proposition 1.3).
\end{remark}

One can ask whether the $\epsilon$ in Theorem \ref{mainsimple} is really necessary. It is in fact, as was shown in \cite{hano};  there exist nonnegative quantum graphs which do not coincide up to labels and isolated vertices with a sum of squares from some $\A_k^0$ or even $\A_k.$ We sketch the idea from \cite{hano}.

 For two graphs $F,G$ we consider a parametrized version of $t(F,G).$ We assign a variable $x_w$ to each of the vertices $w$ of $G$, set $g:=\sum_{w\in V_G} x_w$ and define $$\widetilde t(F,G):=\frac{\sum_{\tiny \begin{array}{c}\varphi\colon V_F\rightarrow V_G\\ {\rm homomorphism}\end{array}} \prod_{v\in V_F}x_{\varphi(v)}} {g^{\vert V_F\vert}}.$$ Note that $\widetilde t(F,G)(1,\ldots,1)=t(F,G)$ is just the usual homomorphism density. Also note that $\widetilde t(\cdot ,G)$ is isolate indifferent. We thus obtain linear  maps  $$\widetilde t(\cdot,G)\colon\A_k\rightarrow \R\left[\frac{x_w}{g} \mid w\in V_G\right]$$ which are compatible with $\boxplus$.  Given $F\in \mathcal G_k$ and a mapping $\psi\colon [k] \rightarrow V_G$  there is a relative version $$\widetilde t_\psi (F,G):= \frac{\sum_{\tiny \begin{array}{c}\varphi\supseteq \psi\\ {\rm homomorphism}\end{array}} \prod_{v\in V_F\setminus [k]}x_{\varphi(v)}} {g^{\vert V_F\vert-k}}.$$ The relative version is multiplicative on $\mathcal G_k$, i.e. $$\widetilde t_\psi (F *_k F',G)=\widetilde t_\psi (F,G)\cdot \widetilde t_\psi (F',G)$$ holds; so $\widetilde t_\psi (\cdot ,G)\colon \A_k\rightarrow \R[x_w/g]$ maps sums of squares to sums of squares. On $\A_k$ we have $$\widetilde t(\cdot, G)= \sum_{\psi\colon [k] \rightarrow V_G} \widetilde t_\psi (\cdot,G) \cdot \frac{\prod_{i\in [k]} x_{\psi(i)}}{g^k}.$$ 
 
  \begin{definition}

 (1) A subset $P\subseteq R$ of a commutative ring $R$ is called a {\it preorder} if $P+ P\subseteq P$, $P\cdot P\subseteq P$ and $P$ contains all squares from $R$.
 
 \noindent
 (2) For $r_1,\ldots,r_m\in R$, the set $${\rm PO}(r_1,\ldots,r_m):=\left\{ \sum_{e\in\{0,1\}^m} \sigma_e\cdot r_1^{e_1} \cdots r_m^{e_m}\mid \sigma_e\in \Sigma R^2\right\}$$ is the smallest preorder containing $r_1,\ldots,r_m$. It is called the preorder {\it generated by} $r_1,\ldots,r_m$.
 \end{definition}
 
\noindent 
 So after clearing denominators in $\tilde t(\cdot,G)$ we get:
\begin{theorem}Let $a$ be a quantum graph, which coincides (after unlabeling and up to isolated vertices) with a sum of squares from some $\A_k$. Then there is some $N$ large enough, such that  for all graphs $G$ $$\left(\sum_{w\in V_G} x_w\right)^{N}\cdot \widetilde t(a,G)\in {\rm PO}(x_w\mid \omega \in V_G)\ \subseteq \   \R[x_w\mid \omega\in V_G].$$  If $a$ coincides with a sum of squares from $\A_k^0$, then $(\sum_{w\in V_G} x_w)^k\cdot \widetilde t(a,G)$ has nonnegative coefficients, for all graphs $G$.
\end{theorem}

\begin{remark}\label{denspos}
The theorem shows that homomorphism densities $t(\cdot,G)$ are reflection positive. We have $t(\cdot,G)=\tilde t(\cdot ,G)(1,\ldots,1)$, and polynomials from the preorder generated by the $x_\omega$ are nonnegative at this point.
\end{remark}

\begin{example}

We have seen in Example \ref{ex1} that the following quantum graph comes from a  sum of squares in $\A_1:$

 $$\begin{tikzpicture}[-,>=stealth',shorten >=1pt,auto,node distance=2.8cm,
                    semithick]
  \tikzstyle{every state}=[fill,draw=none,text=white,scale=0.2]

  \node[state] (A)                    {};
  \node[state]         (B) [below left of=A] {};
  \node[state] (C) [below right of =A]{};

  \path (A) edge           node {} (B)
         edge   []        node {} (C)
               ;

  \node[state]         (D) at (2,0) {};
   \node[state]         (E) [below  of=D] {};
\node[state] (F) [right of = D]{};
\node[state] (G) [below of = F]{};
\path
		(D) edge   []        node {} (E)
  		 (F) edge node {} (G);             
              
    \put (30,-10){$-$}

\put (-50,-8){$a=$}
               \end{tikzpicture}$$ It is also shown in \cite{lo} that $a$ does not come from a sum of squares in some $\A_k^0.$ Here is another proof: for $G=K_2$ we compute $(x_1+x_2)^4\cdot \widetilde t(a,G)=(x_1-x_2)^2x_1x_2$, and this homogeneous polynomial has a zero in the interior of the positive orthant. It can thus clearly not have  the P\'olya property, i.e. multiplication with powers of $x_1+x_2$ will never lead to only nonnegative coefficients (see \cite{capore} for more details on the P\'olya property). 

\end{example}

The paper \cite{hano} uses the described method to show that there even exist nonnegative quantum graphs that are not sums of squares from any $\A_k.$
Now that we have explained the setup for simple graphs in some detail, we pass to multigraphs, and prove some new results.

\section{Multigraphs}\label{multisec}

In this section, a graph is still finite, undirected and loopless, but may now have multiple edges. Note that the case of loops and even directed edges is quite similar, and the results have straightforward extensions.

 We define $k$-labeled graphs and their multiplication as before, except that we don't erase multiple edges after multiplication. All structures as the graph algebras $\A_k$, $\A_k^0$, $\mathcal B_k,\mathcal B_k^0$ and the Reynolds operator $r$ can be defined just as  before. This time $\mathcal A_k^0$ is not finite dimensional, but $\A_k^0=\R[z_{ij}\mid 1\leq i<j\leq k]$ is the full polynomial algebra, and thus $\mathcal V(\A_k^0)=\R^{k\choose 2}$. The algebra $\mathcal B_k^0$ of $S_k$-invariants is finitely generated (by a standard result of Hilbert, see for example \cite{st} for a nice exposition), and to the embedding into $\A_k^0$ there corresponds a polynomial mapping $\mathcal V(\A_k^0)\rightarrow \mathcal V(\mathcal B_k^0).$ We denote by $\mathcal B^0$ the direct limit of the chain $$\cdots\rightarrow\mathcal B_k^0 \stackrel{\boxplus_r}{\rightarrow} \mathcal B_{k+1}^0\rightarrow\cdots $$ again in the category of vector spaces.

We again consider  $\mathcal C_k:= r(\Sigma^2 \A_k^0)=(\Sigma^2 \A_k^0) \cap \mathcal B_k^0$ and this is a preorder of $\mathcal B_k^0$, which is  now larger than $\Sigma^2\mathcal B_k^0$ in general. We still have $\boxplus_r(\mathcal C_{k})\subseteq \mathcal C_{k+1}$ and we obtain a convex cone $\mathcal C=\bigcup_k\mathcal C_k$ in $\mathcal B^0$.  
More general, let $\mathcal P_k\subseteq \A_k^0$ be an $S_k$-invariant preorder. Then $r(\mathcal P_k)=\mathcal P_k \cap \mathcal B_k^0$ is a preorder and a $\mathcal C_k$-module. If $\boxplus(\mathcal P_k)\subseteq \mathcal P_{k+1}$ holds, then also $\boxplus_r(r(\mathcal P_k))\subseteq r(\mathcal P_{k+1}).$  So $\mathcal B^0$ contains the convex cone $\mathcal P=\bigcup_k r(\mathcal P_k).$ If $1$ is an interior point of each $\mathcal P_k$ in $\A_k^0$ (recall this means $1+\epsilon a\in\mathcal P_k$ for all $a$ and $\epsilon$ small; this is sometimes also referred to as $\mathcal P_k$ being {\it archimedean}), then the same is true for $r(\mathcal P_k)$ in $\mathcal B_k^0$ and $\mathcal P$ in $\mathcal B^0$. 
We will mostly consider the preorders $$\mathcal P_k(d):= {\rm PO}(d\pm z_{ij}\mid 1\leq i<j\leq k)\subseteq \A_k^0,$$ of which $1$ is an  interior point (see \cite{ma} or \cite{pede}). The induced cone in $\mathcal B^0$ is denoted by $\mathcal P(d)$ in this case.

 Graph parameters and their properties are defined as before. Furthermore, a  graph parameter is called {\it $d$-bounded}, if $\vert t(K_2^k)\vert \leq d^k$ holds for all $k$, where $K_2^k$ is the graph with two vertices and $k$ edges  between them. With a suitable notion of homomorphism for multigraphs, the homomorphism density $t(\cdot,G)$ into a multigraph $G$ with edge-multiplicity at most $d$ is an example of such a $d$-bounded parameter.
The following Lemma is the straightforward extension of Lemma \ref{help} to the multigraph setting.

\begin{lemma}\label{help2}(1) A family of linear functionals $\varphi\colon \mathcal B_k^0\rightarrow \R$ comes from a flatly reflexion positive and isolate indifferent graph parameter if and only if it is compatible with the embeddings $\boxplus_r$ and satisfy $\varphi(\mathcal C_k)\geq 0$ for all $k$. Equivalently, if it comes from a linear functional  on $\mathcal B^0$ which is nonnegative on $\mathcal C.$

(2) The family comes from a $d$-bounded such parameter, if and only if it comes from a linear functional on $\mathcal B^0$ which is nonnegative on $\mathcal P(d).$\end{lemma}
\begin{proof}Again an exercise. For (2) use the  fact the boundedness just means $\vert \varphi (z_{ij}^k)\vert \leq d^k$ for all $k,i,j$. This is equivalent to having representing measures on $[-d,d]^{k\choose 2}$ for all $k$ (by Theorem 2.2 in \cite{lose1.5} for example), and this is equivalent to being nonnegative on each  $\mathcal P_k(d),$ $r(\mathcal P_k(d))$ and $\mathcal P(d)$, respectively (by \cite{sch}). 
\end{proof}

The following is our first main theorem. Also compare to Theorem \ref{three} below, which provides a more complicated approximation, but avoids the preorder.
\begin{theorem}\label{main}
A quantum multigraph $a$ is nonnegative at each isolate indifferent, flatly reflection positive and $d$-bounded graph parameter if and only if for each $\epsilon >0$ there is some $k$ and some $\sigma\in \mathcal P_k(d)$, such that $a+\epsilon$ and $\sigma$ coincide up to labels and isolated vertices.
\end{theorem}
\begin{proof}
One direction is clear. For the other, let $a$ be nonnegative. Choose some $\tilde a\in\A_d^0$ which coincides with $a$ up to isolated vertices, when all labels are removed. Then $b:=r(\tilde a)\in \mathcal B_d^0$ also coincides with $a$ up to isolated vertices and labels, and is thus nonnegative at each isolate indifferent, flatly reflection positive and $d$-bounded  graph parameter.  In view of Lemma \ref{help2}, it belongs to the double dual of $\mathcal P(d)$ in $\mathcal B^0$, and the isolation theorem for convex sets with nonempty interior again  implies $b+\epsilon \in \mathcal P(d)$ for all $\epsilon>0$. Since $\mathcal B^0$ is the direct limit of the $\mathcal B_k^0$ and $\mathcal P(d)=\bigcup_k r(\mathcal P_k(d))$, this proves the claim.
\end{proof}

\begin{remark}
Again the proof shows that if $a$ is strictly positive at each nontrivial, isolate indifferent, flatly reflexion positive and $d$-bounded parameter, then $a$ comes from an element in some $\mathcal P_k(d)$ without error.
\end{remark}

It is maybe not very surprising that the error $\epsilon$ cannot be removed here as well. To see this, let  $F$ be a multigraph. For any $n\in\N$ we set $g=\sum_{i=1}^n x_i$ and define $$\tilde{t}(F,n):= \frac{\sum_{\varphi\colon V_F\rightarrow [n]} \prod_{v\in V_F}x_{\varphi(v)} \prod_{vw\in E_F} y_{\varphi(v)\varphi(w)}}{g^{\vert V_F\vert}}\in \R\left[\frac{x_i}{g},y_{ij}\mid 1\leq i\leq j\leq n\right].$$ This counts the number of vertex-edge-homomorphisms into the complete graph with vertex weights $x_i$ and edge weights $y_{ij}.$  Again $\tilde{t}(\cdot,n)$ is isolate indifferent and defines linear maps on all $\A_k,$ compatible with $\boxplus$.  
 For $F\in\mathcal G_k$ and $\psi\colon[k]\rightarrow [n]$ there is again a relative version $$\tilde{t}_\psi(F,n)=\frac{\sum_{\varphi\supseteq\psi} \prod_{v\in V_F\setminus[k]} x_{\varphi(v)} \prod_{vw\in E_F} y_{\varphi(v)\varphi(w)}}{g^{\vert V_F\vert -k}}$$ which is multiplicative on $\A_k$ and fulfills $$\tilde{t}(\cdot,n)=\sum_{\psi\colon[k]\rightarrow [n]} \tilde{t}_\psi(\cdot,n) \cdot \frac{\prod_{i=1}^k x_{\psi(i)}}{g^k}.$$
Note that $\tilde{t}_\psi(z_{ij},n)=y_{\psi(i)\psi(j)}$. After clearing denominators we obtain:

\begin{theorem}Let $a$ be a quantum multigraph which coincides up to labels and isolated vertices with an element $\sigma\in {\rm PO}(d\pm z_{ij}\mid 1\leq i<j\leq k)\subseteq \A_k.$ Then there is some $N\in\N$ such that for any $n\in\N$ we have $$\left(\sum_{i=1}^n x_i\right)^N\tilde{t}(a,n)\in {\rm PO}(x_i, d\pm y_{ij}\mid 1\leq i\leq j\leq n)\subseteq \R[x_i,y_{ij}].$$
If $a$ comes from an element of some $\mathcal P_k(d)\subseteq \A_k^0$, then in $g^k\cdot\tilde t(a,n)$  the coefficient of each monomial in $x$ is from ${\rm PO}(d\pm y_{ij})\subseteq \R[y_{ij}].$

\end{theorem}

It is often enough to  substitute $x_i=1/n$ and $y_{ii}=0$ and obtain an element of  the preorder ${\rm PO}(d\pm y_{ij}\mid 1\leq i<j\leq n).$

\begin{example}\label{rob} (1) We consider the Robinson quantum multigraph 

\medskip\begin{center}
\begin{tikzpicture}[-,>=stealth',shorten >=1pt,auto,node distance=2.8cm,
                    semithick]
  \tikzstyle{every state}=[fill,draw=none,text=white,scale=0.2]

  \node[state] (A)                    {};
  \node[state]         (B) [below of=A] {};

  \path (A) edge   [bend right=60]        node {} (B)
         edge   [bend right=15]        node {} (B)
         edge   [bend right=40]        node {} (B)
         edge   [bend left=15]        node {} (B)
         edge   [bend left=40]        node {} (B)
         edge   [bend left=60]        node {} (B)
      ;
   
  \node[state] (C)  at (2,0)              {};
  \node[state]         (D) [below left of=C] {};
   \node[state]         (E) [below right of=C] {};

\path (C) edge   []        node {} (D)
  		edge   [bend right=30]        node {} (D)
              (C) edge   []        node {} (E)
  		edge   [bend left=30]        node {} (E)
		(D) edge   []        node {} (E)
  		edge   [bend right=30]        node {} (E);

                \node[state] (F) at (4,0)                    {};
  \node[state]         (G) [below of=F] {};
   \node[state]         (H) [right of=G] {};

  \path (F)  edge   [bend right=15]        node {} (G)
         edge   [bend right=40]        node {} (G)
         edge   [bend left=15]        node {} (G)
         edge   [bend left=40]        node {} (G)
        (G)  edge   [bend right=20]        node {} (H)
      (G)  edge   [bend left=20]        node {} (H)

      ;

      \put (-40,-10){$a=$};    
   \put (20,-10){$+$};       
   \put (85,-10){$-2$};       
              \end{tikzpicture}
\end{center}

\noindent
which coincides (up to labels, isolated nodes and dividing by $3$) with the fully labeled  graph coming from the Robinson polynomial $$R=z_{12}^6+z_{13}^6+z_{23}^6 -(z_{12}^4z_{13}^2 +z_{12}^4z_{23}^2 +z_{12}^2z_{13}^4 +z_{13}^4z_{23}^2 +z_{12}^2z_{23}^4 +z_{13}^2z_{23}^4) +3z_{12}^2z_{13}^2z_{23}^3\in \A_3^0.$$ For details on the Robinson polynomial see \cite{re}. Since the Robinson polynomial is nonnegative on $\R^3$, we have $R+\epsilon \in \mathcal P_3(d)\subseteq \A_3^0$ for all values of $d$. This follows for example from the archimedean Positivstellensatz in \cite{sch}.   Thus $a$ is nonnegative at each $d$-bounded, flatly reflexion positive and isolate indifferent graph parameter. 

On the other hand, if we compute $\tilde t(a,3)$ and set $x_i=1/3$ as well as $y_{ii}=0$ for all $i$, then we obtain $R$ again (up to a positive multiple and in the variables $y_{ij}$ instead of $z_{ij}$). Since $R$ is homogeneous and not a sum of squares, it does also not belong to the preorder generated by $d\pm y_{ij}$ (compare the lowest degree parts in a possible representation). So  $a$ does not coincide up to labels and isolated nodes with an element from some ${\rm PO}(d\pm z_{ij})$ in $\A_k$ (and thus also not from some $\mathcal P_k(d)\subseteq \A_k^0$).

 Reducing all edge multiplicities in the Robinson example to one yields the simple quantum graph from Example \ref{quex} (2), and thus proves its nonnegativity. In particular, it implies Goodman's Theorem.

(2) Several generalizations of the Robinson polynomial appear under the name $H_\mu$ in \cite{chla}, Remark 2.5 and Proposition 2.7. They can be used to produce generalizations of the above example. For any odd integer $\mu$ we obtain the following nonnegative quantum graph, where the little numbers indicate the multiplicities of the simply drawn edge:

\bigskip
\begin{center}
\begin{tikzpicture}[-,>=stealth',shorten >=1pt,auto,node distance=2.8cm,
                    semithick]
  \tikzstyle{every state}=[fill,draw=none,text=white,scale=0.2]

  \node[state] (A)                    {};
  \node[state]         (B) [below of=A] {};

  \path (A) edge   []        node {} (B)
            ;
    \put (2,-10){\tiny $2\mu +4$}; 
  \node[state] (C)  at (2.5,0)              {};
  \node[state]         (D) [below left of=C] {};
   \node[state]         (E) [below right of=C] {};

\path (C) edge   []        node {} (D)
  		
              (C) edge   []        node {} (E)
  		edge   [bend left=30]        node {} (E)
		(D) edge   []        node {} (E)
  		edge   [bend right=30]        node {} (E);
       \put (55,-5){\tiny\bf $2\mu$}       
              
                \node[state] (F) at (4.5,0)                    {};
  \node[state]         (G) [below of=F] {};
   \node[state]         (H) [right of=G] {};

  \path (F)  edge   []        node {} (G)
               (G)  edge   [bend right=20]        node {} (H)
      (G)  edge   [bend left=20]        node {} (H)

      ;

      \put (-40,-10){$a=$};    
   \put (35,-10){$+$};       
   \put (95,-10){$-2$};       
    \put (130,-8){\tiny $2\mu +2$}; 
              \end{tikzpicture}
\end{center}

(3) Another related polynomial  appears under the name $h_4$ in \cite{chla}, Section 2. It gives rise to the following nonnegative quantum graph:

\medskip\begin{center}
\begin{tikzpicture}[-,>=stealth',shorten >=1pt,auto,node distance=2.8cm,
                    semithick]
  \tikzstyle{every state}=[fill,draw=none,text=white,scale=0.2]

  \node[state] (A)                    {};
  \node[state]         (B) [below of=A] {};

  \path (A) edge   [bend right=60]        node {} (B)
         edge   [bend right=15]        node {} (B)
         edge   [bend right=40]        node {} (B)
         edge   [bend left=15]        node {} (B)
         edge   [bend left=40]        node {} (B)
         edge   [bend left=60]        node {} (B)
      ;
   
  \node[state] (C)  at (2,0)              {};
  \node[state]         (D) [below left of=C] {};
   \node[state]         (E) [below right of=C] {};

\path (C) edge   [bend right=15]        node {} (D)
  		edge   [bend right=40]        node {} (D)
edge   [bend left=30]        node {} (D)
edge   [bend left=10]        node {} (D)

              (C) edge   []        node {} (E)
  		
		(D) edge   []        node {} (E)
  	;

                \node[state] (F) at (4,0)                    {};
  \node[state]         (G) [below of=F] {};
   \node[state]         (H) [right of=G] {};

  \path (F)  edge   [bend right=25]        node {} (G)
         edge   [bend right=45]        node {} (G)
         edge   [bend left=25]        node {} (G)
         edge   [bend left=45]        node {} (G)
edge   []        node {} (G)

        (G)  edge   []        node {} (H)

      ;

      \put (-40,-10){$a=$};    
   \put (20,-10){$+$};       
   \put (85,-10){$-2$};       
              \end{tikzpicture}
\end{center}

\end{example}

%
%
%
%
%
%
%

We proceed and want to prove another Positivstellensatz. Let us call a graph parameter $t$ {\it slowly growing}, if $$\sum_{i=0}^\infty \frac{1}{i!} t\left(K_2^{2i}\right)<\infty$$ where again $K_2^{j}$ is the graph with two vertices and $j$ edges between them.

\begin{theorem}\label{weakpos} A quantum multigraph $a$ is nonnegative at each isolate indifferent, flatly reflection positive and slowly growing graph parameter, if and only if for all $\epsilon >0$ there exists $r\in \N$  such that  $$a+\epsilon \sum_{i=0}^r \frac{1}{i!} K_2^{2i}$$ coincides with a sum of squares from $\A_r^0$, up to labels and isolated nodes.
\end{theorem}
\begin{proof} The ''if''-direction is clear. For the ''only if''-direction we can assume that $a$ is strictly positive at each normalized such parameter (i.e. $t(K_1)=1$), by adding some $\epsilon>0$ to $a$ first. 

We consider the finite dimensional subspace $V_k=\R[z_{ij}]_k\subseteq \A_k^0$ of polynomials of degree at most $k$, set $\Sigma^2 V_k=\left\{\sum_i c_i^2\mid c_i\in V_k\right\}$, and finally $$\Sigma_k:= r(\Sigma^2 V_k)\subseteq\mathcal B_k^0.$$ This is a convex cone in a finite dimensional subspace of $\mathcal B_k^0.$ For any fixed $M\geq 1$ we next consider  $$\mathcal K_k(M):=\Sigma_k+\R_{\geq 0}\left( {k\choose 2}\cdot M-\sum_{1\leq i<j\leq k} \sum_{s=0}^k \frac{1}{s!} z_{ij}^{2s}\right),$$ which is also a finite dimensional cone in  $\mathcal B_k^0$. We have $\boxplus_r(\mathcal K_k(M))\subseteq \mathcal K_{k+1}(M).$

As usual, we choose some $b\in\mathcal B_d^0$ that coincides with $a$ up to labels and isolated nodes. We then claim that $\boxplus_r(b)$ belongs to $\mathcal K_k(M)$, for some $k$ large enough. Indeed if it does not, there are  $\mathcal K_k(M)$-positive functionals $\varphi_k\colon\mathcal B_k^0\rightarrow \R$ with $\varphi_k(b)\leq 0$ for all $k$ large enough. We can ensure $\varphi_k(1)> 0$ (and thus $\varphi_k(1)=1$): if $b$ is not in the linear hull of $\mathcal K_k(M)$ then first choose $\varphi_k\equiv 0$ on $\mathcal K_k(M)$ and negative on $b$, then add some small multiple of the evaluation at the origin; otherwise choose $\varphi_k$ nontrivial on $\mathcal K_k(M)$, and use Lemma 4.3 from \cite{lane} to see that $\varphi_k(1)\neq 0$ is automatic. 

Again using Lemma 4.3 from \cite{lane}Ê one checks that $\varphi_s(\boxplus_r(c))$ remains bounded for each fixed $c$ from some $\mathcal B_k^0$ and all $s.$ Choosing a non-principal ultrafilter $\omega$ on $\N$ and setting $$\psi_k(c):=\lim_{s\to\omega}\varphi_s(c)$$ for all $k$ and $c\in\mathcal B_k^0$ defines a new compatible family of linear functionals, nonnegative on all $\mathcal C_k$. This family thus comes from a normalized, flatly reflection positive and isolate indifferent graph parameter $t$, which is obviously slowly growing, in fact $$\sum_{i=0}^\infty \frac{1}{i!} t\left(K_2^{2i}\right)\leq M$$ holds. We also have $t(a)\leq 0$, a contradiction.

What we have shown so far is that for each $M\geq 1$ there exists some $k$ large enough such that $b\in\mathcal K_k(M).$ This means we find a representation $$b +c\sum_{1\leq i<j\leq k}\sum_{s=0}^{k} \frac{1}{s!} z_{ij}^{2s}=\sigma +c{k\choose 2}M$$ with some $c\geq 0$ and $\sigma\in\Sigma_k.$ If we plug in $0$ for each $z_{ij}$ and let $M$ go to infinity, we see $c{k\choose 2} \to 0.$ This proves the claim.
\end{proof}

\begin{remark}(1) Theorem \ref{weakpos} gives an explicit $\ell_1$-norm approximation of $a$ via sums of squares. In this setup, the approximation cannot be strengthened to a simple "$+\epsilon$'' approximation, as we will see.

(2)
From the main result of \cite{la} we see that  a globally nonnegative polynomial $p\in\R[z_{ij}]=\A_k^0$ gives rise to a quantum graph that is nonnegative in the sense of Theorem \ref{weakpos}.

(3) Whether the perturbation is really necessary is checked as before; if $\tilde t(a,n)$ is not a sum of squares (after setting $x_i=1/n$ and $y_{ii}=0$ often), then $a$ does not coincide with a sum of squares from some $\A_k$.
\end{remark}

\begin{example}
The Robinson example $a$ from Example \ref{rob} (1) is nonnegative in the sense of Theorem \ref{weakpos}, since it comes from a globally nonnegative polynomial in $\A_3^0$. As argued before, neither $a$ nor $a+\epsilon$ is a sum of squares in some $\A_k$,  since $\tilde t(a,3)$ is the Robinson polynomial again.
\end{example}

In a similar fashion, we can prove the following variant of Theorem \ref{main}. We get a more complicated approximation, but avoid the preorder:

\begin{theorem}\label{three}
 A quantum multigraph $a$ is nonnegative at each isolate indifferent, flatly reflection positive and $d$-bounded graph parameter if and only if  for all $\epsilon>0$ there is some $r$ such that $$a+\epsilon  \left(1+\frac{1}{d^{2r}}K_2^{2r}\right)$$ coincides with a sum of squares from $\A_r^0$, up to labels and isolated nodes.
\end{theorem}
\begin{proof} By scaling the edge-weights we can restrict to the case $d=1.$
The "if"-direction is clear. For the other direction we again assume that $a$ is strictly positive at each normalized such parameter (this is why we need $1$ in the approximation).  We proceed as in the proof of Theorem \ref{weakpos}, this time setting $$\mathcal K_k(M):=\Sigma_k+\R_{\geq 0}\left( {k\choose 2}\cdot M-\sum_{1\leq i<j\leq k} z_{ij}^{2k}\right).$$ Using Lemma 4.1 and Lemma 4.3 from \cite{lane} we obtain $b\in\mathcal K_k(M)$ for some $k$. Note that the functionals $\psi_k$ that we define as before fulfill $\psi_k(z_{ij}^{2r})\leq M$ on $\mathcal A_k^0$; by Theorem 2.5 in Chapter 4 of \cite{bechre} they have representing measures on $[-1,1]^{{k\choose 2}}$ and thus lead to a $1$-bounded parameter.
We obtain representations $$b+c\sum_{1\leq i<j \leq k}z_{ij}^{2k}= \sigma + c{k\choose 2}M$$ and again $c{k\choose 2}$ goes to zero for $M\to \infty$.\end{proof}

\end{document}